\documentclass[reqno, 10pt]{amsart}

\usepackage{amsmath,amssymb,amsthm,tikz}

\newtheorem*{thm}{Theorem}
\newtheorem*{lemma}{Lemma}

\begin{document}

\title[]{On an iterated arithmetic function\\ problem of Erd\H{o}s and Graham}

\author[]{Stefan Steinerberger}

\address{Department of Mathematics, University of Washington, Seattle, WA 98195, USA}
 \email{steinerb@uw.edu}

 
\begin{abstract} Erd\H{o}s and Graham define $g(n) = n + \phi(n)$ and the iterated application $g_k(n) = g(g_{k-1}(n))$. They ask for solutions of $g_{k+r}(n) = 2 g_{k}(n)$ and observe $g_{k+2}(10) = 2 g_{k}(10)$ and $g_{k+2}(94) = 2 g_{k}(94)$. We show that understanding the case $r = 2$ is equivalent to understanding all solutions of the equation $\phi(n) + \phi(n + \phi(n)) = n$ and find the explicit
solutions $ n = 2^{\ell} \cdot \left\{1,3,5,7,35,47\right\}$. This list of solutions is possibly complete: any other solution derives from a number $n=2^{\ell} p$ where $p \geq 10^{10}$ is a prime satisfying $\phi((3p-1)/4) = (p+1)/2$.  Primes with this property seem to be very rare and maybe no such prime exists.
\end{abstract}
\maketitle

\section{Introduction and Result}

In their 1980 collection of \textit{Old and new problems and results in combinatorial number theory} \cite{erd}, Erd\H{o}s and Graham write

\begin{quote}
\textit{Let $g(n) = n + \phi(n) = g_1(n)$ and let $g_{k}(n) = g(g_{k-1}(n))$ for $k \geq 2$. For which $n$ is $g_{k+r}(n) = 2 g_k(n)$ for all (large) $k$?  The known solutions to $g_{k+2}(n) = 2 g_k(n)$ are $n=10$ and $n=94$.}
\end{quote}

It also appears as problem $\# 411$ in Bloom's collection \cite{bloom}. The motivation behind the problem is not entirely clear. It combines $\phi(n)$, an object determined by multiplicative properties of the integers, with addition and iteration. Such problems tend to be fairly intractable. Somewhat to our surprise, the case $r = 2$ has just barely enough rigidity that it can be almost completely solved: we identify the structure of such solutions, describe six infinite families and show that no other `small' solutions exist.  It is conceivable that this list of six solutions is complete.

\begin{thm}
The problem is equivalent to finding solutions of $\phi(n) + \phi(n + \phi(n)) = n$.  If $n \in \mathbb{N}$ solves this equation, then, for some $\ell \in \mathbb{N}_{ \geq 1}$
\begin{enumerate}
\item either $ n \in 2^{\ell} \cdot \left\{1, 3, 5, 7, 35, 47 \right\}$ or
\item  $ n = 2^{\ell} \cdot (8m+7)$ or $n = 2^{\ell} \cdot (6m+5)$ where $8m+7 \geq 10^{10}$ is a prime number and $\phi(6m + 5) = 4m + 4$.
\end{enumerate}
\end{thm}

Whether the list of solutions (1) is complete depends on whether there exist additional primes of the form $p = 8m+7$ such that $\phi(6m + 5) = 4m + 4$.  Dropping the primality constraint, one may ask about the more general question whether 
$$ \prod_{p |n} \left(1 - \frac{1}{p} \right) = \frac{\phi(n)}{n} = \frac{2}{3} + \frac{2}{3n} \qquad \mbox{has infinitely many solutions}.$$
Any solution of $\phi(6m + 5) = 4m + 4$ gives rise to a solution with $n=6m+5$.
This problem has the solutions $n=5, 35, 1295, 1679615$ which have amusing prime factorizations
\begin{align*}
35 &= 5 \cdot 7 \\
1295 &= 5 \cdot 7 \cdot 37\\
1679615 &= 5 \cdot 7 \cdot 37 \cdot 1297.
\end{align*}
The equation does not seem to have any other `small' solutions; and maybe no other solutions at all. Since the product over all primes converges to 0 
$$ \frac{\phi(n)}{n} = \prod_{p |n} \left(1 - \frac{1}{p} \right)$$
can get arbitrarily close to any number in $[0,1]$. The question here is
thus related to whether there are finite subsets of the primes for which the product just happens to be extraordinarily close to $2/3$ given the number of different primes involved.\\

\textbf{The bigger picture.} Erd\H{o}s and Graham additionally remark that \textit{Selfridge and Weintraub both found solutions to $g_{k+9}(n) = 9 g_k(n)$ (all $n$ found were even) and Weintraub also discovered
$$ g_{k+25}(3114) = 729 \cdot g_k(3114), \quad k \geq 6.$$
We know of no general rules for forming such examples}. A quick search shows that there are many such examples. The dynamics does seem to be intricate but not without structure: the numbers that eventually seem to satisfy $g_{k+9}(n) = 9 g_k(n)$ include $n=130, 170, 234, 260, 266, \dots$. Additional curious examples include $g_{k+20}(385) = 6561 \cdot g_k(385)$ as well as $g_{k+25}(1570) = 729 \cdot g_{k}(1570)$ and $g_{k+25}(1702) = 729 \cdot g_{k}(1702)$. Many more such solutions exist, they mostly seem to correspond to shifts by $r= 9$ and $r=25$. Shifts by $r=20$ exist but are less frequent. A rare outlier is
$$ g_{k+14}(3393) = 729 \cdot g_k(3393)$$
together with $ g_{k+14}(6175) = 729 \cdot g_k(6175)$ and $g_{k+14}(6969) = 729 \cdot g_k(6969)$.
Maybe the existence of many such solutions is not too surprising: one might argue that for `most' integers one has $\phi(n) \ll n$ and thus $g(n) \sim n$. The only way to escape from this trap and to reach a multiple of $n$ within $r$ steps is to either have very few small prime factors or only relatively large prime factors (that's the only way for $\phi(n)$ to be comparable to $n$).  Having relatively large prime factors is difficult to maintain, multiples of small primes are everywhere. Then, however, if we are forced to having relatively few small primes (which is a common theme for virtually all of the examples), `coincidences' along this type may not be unavoidable. Once one has a single solution of $g_{k+r}(n) = g_{k}(n)$, extending to an infinite family of equations for all
$k \geq k_0$ may just be a consequence of a basic divisibility property (that we also exploit in \S 2.1).

\section{Proof}

\subsection{Equivalence.} We start by making the equivalence between $g_{k+2}(n) = 2 g_k(n)$ and $\phi(m) + \phi(m + \phi(m)) = m$ precise.  Let us suppose $g_{k+2}(n) = 2 g_{k}(n)$ for
some $n$ and $k$. Introducing the abbreviation $m = g_k(n)$, we have
$$g_2(m) =  m + \phi(m) + \phi(m + \phi(m)) = 2m$$
which is exactly the Diophantine equation under consideration.  Conversely, let us now suppose that we have a solution of this Diophantine equation. $m=1,2$ are not solutions and $\phi(m)$ is even for $m \geq 3$. This forces $m$ to be even: if it were odd, the left-hand side would be odd while the right-hand side would be even. However, for even numbers $m$, one has $\phi(2m) = 2 \phi(m)$. Then, however,
$$ g(2m) = 2m + \phi(2m) = 2m + 2 \phi(m) = 2 g(m).$$
Therefore, if one has a solution of $ g_{k+2}(n) = 2 g_{k}(n)$, then $g_{k}(n)$ has to be even and the equation is automatically satisfied for all subsequent values of $k$ since
$$ g_{k+3}(n) = g ( g_{k+2}(n)) = g (2 g_k(n)) = 2 g(g_k(n)) =2 g_{k+1}(n).$$

\subsection{Powers of 2} If  $n=2^k$, then $\phi(n) = 2^{k-1}$ and $n + \phi(n) = 3 \cdot 2^{k-1}$ and 
$\phi(n + \phi(n)) = \phi( 3 \cdot 2^{k-1}) = 2^{k-1}$ and
$$ n + \phi(n) +  \phi(n + \phi(n)) =3 \cdot 2^{k-1} + 2^{k-1} = 2^{k+1} = 2n.$$
Therefore each power of 2 is a solution. More, precisely, we get the picture
$$ 2^k  \xrightarrow {~ g ~ }  3 \cdot 2^{k-1}   \xrightarrow {~ g ~ }  2^{k+1}   \xrightarrow {~ g ~ }  3 \cdot 2^{k}  \xrightarrow {~ g ~ }  2^{k+2}   \xrightarrow {~ g ~ }  \dots$$
which generates the first two infinite families of solutions, these being $2^{k}$ and $3 \cdot 2^{k}$.   

\subsection{A $2-$adic Lemma.}  The advantage of having excluded powers of 2 lies in being able to use the following completely elementary Lemma.
\begin{lemma}
If $n = 2^{\ell} q$ for some $\ell \geq 1$ and some $q > 1$ odd, then $\phi(n)$ has \textit{at least} as many factors of 2 as $n$ with equality iff $q = p^{\alpha}$ for some prime $p \equiv 3 \mod 4$ and $\alpha \in \mathbb{N}_{\geq 1}$. If $n=q \geq 3$ is odd, then $\phi(n)$ is always even and only \emph{not} divisible by 4 when $n= p^{\alpha}$ is a prime power and $p \equiv 3 \mod 4$.
\end{lemma}
\begin{proof}
Writing $n = 2^{\ell} q$ with $q$ being a product of odd primes, we have
$$ \phi\left( 2^{\ell} \prod_{i=1}^{k} p_i^{\alpha_i} \right) =  2^{\ell -1} \prod_{i=1}^{k} \left[ (p_i-1) p_i^{\alpha_i-1} \right].$$ 
Each odd prime $p_i$ contributes an even number $p_i -1$ to the product. The only way to contribute only one even number is to have a single prime power, i.e. $k=1$. If $p_1-1$ contributes a single factor of $2$ then $p_i \equiv 3 \mod 4$.
\end{proof}
 
Suppose now that
$$n + \phi(n) + \phi(n + \phi(n)) = 2n$$
and that $n$ is not a power of 2. We write $n = 2^{\ell_1} a$ for some odd $a \geq 3$ as well as $\phi(n) = 2^{\ell_2} b$ for some odd $b \geq 1$. The Lemma implies that $\ell_2 \geq \ell_1$. This leads to a natural case distinction: 
$$ n + \phi(n) = \begin{cases} 
2^{\ell_1} \cdot c \qquad &\mbox{if}~\ell_2 > \ell_1 \\
2^{\ell_3} \cdot c \qquad &\mbox{if}~\ell_2 = \ell_1,
\end{cases}$$
for some odd $c \geq 1$ and $\ell_3 > \ell_1 = \ell_2$. We quickly show that the case $c = 1$ can be excluded. If $n + \phi(n) = 2^{\ell}$, then 
$$ 2n = n+ \phi(n) + \phi(n + \phi(n)) = 2^{\ell} + \phi(2^{\ell}) = 3 \cdot 2^{\ell -1}$$
and thus $n = 3 \cdot 2^{\ell - 2}$ which is the solution we already found above.  We see that $c=3$ ends up reducing to the same cycle, we may thus assume that $c \geq 5$.
We will treat both cases, $\ell_2 > \ell_1$ and $\ell_2 = \ell_1$, separately (though, having dealt with the first, the second ends up being fairly similar).

\subsection{The case $\ell_2 > \ell_1$, Overview.}
 Let us now first assume that  $n = 2^{\ell_1} a$ for some odd $a > 1$ and that $\phi(n) = 2^{\ell_2} b$ for some odd $b$ and $\ell_2 > \ell_1$. Then
$$ n + \phi(n) =  2^{\ell_1}a + 2^{\ell_2} b = 2^{\ell_1}(a + 2^{\ell_2 - \ell_1} b)$$
which has exactly $\ell_1$ factors of $2$.
Then
$$ \phi(n + \phi(n)) = 2^{\ell_1 - 1} \phi(a + 2^{\ell_2 - \ell_1} b).$$
and the equation $\phi(n) + \phi(n+\phi(n)) = n$ can be written as
$$ 2^{\ell_2}  b + 2^{\ell_1-1}  \phi(a + 2^{\ell_2 - \ell_1} b) = 2^{\ell_1} a.$$ 
Now we count powers of $2$.  Since $a + 2^{\ell_2 - \ell_1} b$ can be assumed to be an odd number $\geq 5$ (otherwise we recover one of the two known solutions), the expression $ \phi(a + 2^{\ell_2 - \ell_1} b) $ is even. Since $\ell_2 > \ell_1$, we need that $ \phi(a + 2^{\ell_2 - \ell_1} b) $ is even but not divisible by $4$, otherwise the left-hand side would be divisible by $2^{\ell_1 + 1}$ while the right-hand side is not. Appealing to the Lemma forces $a + 2^{\ell_2 - \ell_1} b = p^{\alpha}$ for some prime $p \equiv 3 \mod 4$ and some $\alpha \geq 1$. Thus, abbreviating for simplicity $\ell = \ell_1$, we see that
$$ n + \phi(n) = 2^{\ell} p^{\alpha}.$$
This implies $\phi(n + \phi(n)) = 2^{\ell-1} (p-1) p^{\alpha -1}$.
Using the equation, we deduce
  \begin{align*}
  2n &= (n+ \phi(n)) + \phi(n + \phi(n)) =  2^{\ell} \cdot p^{\alpha} +  2^{\ell-1} (p-1) \cdot p^{\alpha-1} \\
  &= 2^{\ell -1} p^{\alpha - 1}(2p + (p-1)) = 2^{\ell - 1} p^{\alpha -1} (3p - 1)
 \end{align*}
 which implies $n = 2^{\ell - 2} (3p-1) p^{\alpha - 1}$. The equation $n + \phi(n) = 2^{\ell} p^{\alpha}$
gives
 $$ 2^{\ell - 2} (3p-1) p^{\alpha - 1} + \phi\left( 2^{\ell - 2} (3p-1) p^{\alpha - 1} \right) = 2^{\ell} \cdot p^{\alpha}.$$
The next steps depend on whether $\alpha = 1$ or $\alpha \geq 2$, we treat them separately.

\subsection{The case $\ell_2 > \ell_1$, Case $\alpha=1$.} If $\alpha = 1$, the equation simplifies to
   $$ 2^{\ell - 2} (3p-1) + \phi\left( 2^{\ell - 2} (3p-1) \right) = 2^{\ell} \cdot p.$$
We write $3p-1 = 2^k q$ where $q$ is odd. Since $p \equiv 3 \mod 4$, we deduce $3p -1 \equiv 8 \mod 12$ and $3p-1$ is thus divisible by $4$ and $k \geq 2$. With this substitution, we have
$$ \phi\left( 2^{\ell - 2} (3p-1) \right) = \phi\left( 2^{\ell + k - 2} q\right) = 2^{\ell+k-3} \phi\left(  q\right)$$
and the equation turns into
    $$ 2^{\ell  + k - 2} q +2^{\ell+k-3} \phi\left(  q\right) = 2^{\ell} \cdot  \frac{2^k q + 1}{3}.$$
 Dividing by $2^{\ell}$ simplifies this to
    $$ 2^{k - 2} q +   2^{ k - 3}  \phi\left(q \right) =   \frac{2^k q + 1}{3}.$$
    We quickly show that $q \geq 3$. q is odd. If $q=1$, then  $\alpha = 1$ and $n= 2^{\ell-2}(3p-1) p^{\alpha -1}$
    would mean that $n$ is a power of $2$ (which leads us to a solution we already know). We may thus assume that $q \geq 3$ which implies that $\phi(q)$ is even.   We note that $2^k q + 1$ is odd and therefore the right-hand side is odd. If it were the case that $k \geq 3$, then the left-hand side would be even because $\phi(q)$ is. Therefore $k=2$ and the equation  further simplifies to
    $$ q +   \frac{1}{2}  \phi\left(q \right) =   \frac{4 q + 1}{3} \qquad \mbox{and thus} \qquad \phi(q) = \frac{2}{3} (q+1).$$
  This equation will be further analyzed in \S 2.10.

\subsection{The case $\ell_2 > \ell_1$, Case $\alpha \geq 2$.} This case deals with
    $$ 2^{\ell - 2} (3p-1) p^{\alpha - 1} + \phi\left( 2^{\ell - 2} (3p-1) p^{\alpha - 1} \right) = 2^{\ell}  p^{\alpha}.$$
    We will show that this has no solutions when $\alpha \geq 2$. As above, we write $3p - 1 = 2^k \cdot q$ where $q > 1$ is odd. Since $p \equiv 3 \mod 4$, we have $3p - 1 \equiv 8 \mod 12$ and therefore $k\geq 2$. We observe that $3p - 1$ and $p$ are coprime and therefore so are $q$ and $p$. This leads to the simplification
 $$  \phi\left( 2^{\ell - 2} (3p-1) p^{\alpha - 1} \right) =  \phi\left( 2^{\ell - 2 + k} q p^{\alpha - 1} \right) = 2^{\ell - 3 + k} \phi(q) (p-1) p^{\alpha - 2}.$$
 Substituting this into the equation, we obtain
 \begin{align*}
 2^{\ell} p^{\alpha} &= 2^{\ell - 2} 2^k q p^{\alpha -1} + 2^{\ell - 3 + k} \phi(q) (p-1) p^{\alpha -2} \\
 &= 2^{\ell} p^{\alpha} \left( 2^{k-2} \frac{q}{p} + 2^{k-3} \frac{(p-1) \phi(q)}{p^2} \right) \\
 &= 2^{\ell} p^{\alpha} \left(   \frac{ 2^{k-2} q p +  2^{k-3} (p-1) \phi(q)}{p^2} \right).
 \end{align*}
We arrive at the equation 
 $$ 2^{k-2} q p + 2^{k-3} (p-1) \phi(q) = p^2$$
 which can also be written as
 $$ 2^{k-3} \cdot (p-1) \cdot \phi(q)  = p \cdot (p - 2^{k-2} q).$$
 Since $\phi(q) \leq q = (3p-1)/2^k \leq (3p-1)/4 < p$, no factor on the left-hand side is a multiple of $p$ while the right-hand side is. Thus no such solution can exist.

\subsection{The case $\ell_2 = \ell_1$, Overview.}
Our goal is to find solutions of
$$n + \phi(n) + \phi(n + \phi(n)) = 2n.$$
We can now assume that $n = 2^{\ell} a$ for some odd $a > 1$ as well as $\phi(n) = 2^{\ell} b$ for some odd $b > 1$ (the case where $b=1$ was already dealt with above and leads to powers of 2). 
Appealing to the Lemma, the only way that $n$ and $\phi(n)$ can have the same number of powers of $2$ is when $n$ is of the form
$$ n = 2^{\ell} p^{\alpha}$$
with some prime $p \equiv 3 \mod 4$ and $\alpha \geq 1$. The remainder of the argument is quite similar to the argument above but we give all the details for clarify of exposition. If $ n = 2^{\ell} p^{\alpha}$, then we have
\begin{align*}
\phi(n) &= 2^{\ell -1} (p-1) p^{\alpha - 1}\\
 n + \phi(n) &=  2^{\ell} p^{\alpha} + 2^{\ell -1} (p-1) p^{\alpha - 1} = 2^{\ell -1} p^{\alpha -1}\left(3p - 1 \right).
\end{align*} 
Motivated by the previous argument, we again rewrite $3p - 1 = 2^k q$ with $q$ odd. As before, we have $k \geq 2$. We observe again that $q$ and $p$ are coprime.  Then
$$ \phi(n + \phi(n)) = \phi( 2^{\ell + k -1} p^{\alpha -1}q) = 2^{\ell + k- 2} \phi(p^{\alpha -1}) \phi(q).$$
 As above, distinguish the cases $\alpha = 1$ and $\alpha > 1$.
\subsection{The case $\ell_2 = \ell_1$, Case $\alpha = 1$.} 
We first show that we may assume that $q \geq 3$.  If $q=1$, then $3p -1 = 2^k$ and $n + \phi(n) = 2^{l + k -1} p^{\alpha - 1}$. 
Since $\alpha = 1$, then $n + \phi(n) = 2^{l +k - 1}$ and $\phi(n + \phi(n)) = 2^{l +k - 2}$ and
$$2n = (n + \phi(n)) + \phi(n + \phi(n)) =  2^{l +k - 1} + 2^{l +k - 2} = 3 \cdot 2^{l +k - 2}$$
which get us to a solution that we already know (the powers of 2). We may now assume that $q \geq 3$ and $q$ odd. 
Then
\begin{align*}
2n = 2^{\ell+1} p &= 2^{\ell -1} p^{\alpha -1}\left(3p - 1 \right) + \phi( 2^{\ell + k -1} p^{\alpha -1}q)  \\
&= 2^{\ell + k -1} q + 2^{\ell + k - 2} \phi(q).
\end{align*}
Since $q \geq 3$ is odd, $\phi(q)$ is even and the left-hand side is divisible by $2^{\ell + k -1}$.  Since $k \geq 2$, this forces $k=2$. Therefore
$3p - 1 =4q$ and
$$  2^{\ell+1} q +  2^{\ell} \phi\left(  q\right) = 2^{\ell+1} \cdot \frac{4q + 1}{3}$$
which requires the same equation as above $ \phi\left(  q\right)  = (2/3) (q+1)$ which will be discussed in \S 2.10.

\subsection{The case $\ell_2 = \ell_1$, Case $\alpha \geq 2$.} We have $\alpha \geq 2$.  Then, since 
\begin{align*}
2^{\ell + 1} p^{\alpha} &= 2n = n + \phi(n) + \phi(n + \phi(n)) \\
&=2^{\ell +k -1} p^{\alpha -1}q + 2^{\ell + k- 2} (p-1) p^{\alpha -2} \phi(q)
\end{align*} 
where $3p-1 = 2^k q$ for $q$ odd.
We first discuss the case $q=1$. Things simplify to
\begin{align*}
 2^{\ell + 1} p^{\alpha} &= 2^{\ell +k -1} p^{\alpha -1} + 2^{\ell + k- 2} (p-1) p^{\alpha -2} \\
 &=  2^{\ell +k -2} p^{\alpha -2} \left( 3p -1 \right)
 \end{align*}
and since $3p-1$ is a power of $2$, we get a contradiction when counting the powers of $p$. We may thus assume that
$q \geq 3$ and that
$$ 2^{\ell + 1} p^{\alpha}  =2^{\ell +k -1} p^{\alpha -1}q + 2^{\ell + k- 2} (p-1) p^{\alpha -2} \phi(q).$$

 Then $\phi(q)$ is even, the right-hand side has more factors of $2$ than the left-hand side unless $k \leq 2$. Since $k \geq 2$, we have $k=2$ and the equation simplifies to, after dividing by $2^{\ell}$ to
$$ 2 p^{\alpha} = 2 p^{\alpha -1}q +  (p-1) p^{\alpha -2} \phi(q).$$
Dividing by $p^{\alpha - 2}$, we have
$$ 2p^2 = 2pq + (p-1) \phi(q)$$
As before, we see that neither $p-1$ nor $\phi(q)$ can be a multiple of $p$ and this concludes the argument.

\subsection{The equation $\phi(q) = (2/3)(q+1)$.}
It remains to analyze the equation
$$ \phi(q) = \frac{2}{3}(q+1)$$
which we obtained in both cases.  In the first case, we had deduced that
$$n = 2^{\ell} \frac{3p-1}{4} $$
for some prime $p \equiv 3 \mod 4$ with the property that $(3p-1)/4 = q$ is an odd number satisfying
$$\phi(q) = \frac{2}{3}(q+1).$$
Note that 
$$ \frac{3p - 1}{4} = q \qquad \mbox{is equivalent to} \qquad p = \frac{4q + 1}{3}$$
for some odd $q$. Since $ p = 4s + 3$ for some $s \in \mathbb{N}$ we have
$$ 3p - 1 = 12s + 8 = 4q \implies 3s + 2 = q$$
which is only odd if $s$ is odd. Replacing $s$ by $2s+1$, we have $q = 3(2s+1) + 2 = 6s + 5$. 
Since $p = (4q + 1)/3$, we get that $p = (24s + 21)/3 = 8s + 7$. Making the ansatz $q = 6m + 5$
leads to the equation $ \phi(6m + 5) =  4m + 4$. Altogether, what is required is $m \in \mathbb{N}_{\geq 0}$ such that
$$ \begin{cases}8m + 7 \qquad \mbox{is prime}\\
\phi(6m + 5) = 4m + 4.
\end{cases}$$
The only solutions that we found were $m=0$ and $m=5$ which correspond to the prime numbers $7$ and $47$. These give rise to the orbits 
$$ 7 \cdot 2^k  \xrightarrow {~ g ~ }  5 \cdot 2^{k+1}   \xrightarrow {~ g ~ }  7 \cdot 2^{k+1}   \xrightarrow {~ g ~ }  5 \cdot 2^{k+2}  \xrightarrow {~ g ~ }  7\cdot 2^{k+2}   \xrightarrow {~ g ~ }  \dots$$
as well as
$$ 47 \cdot 2^k  \xrightarrow {~ g ~ }  35 \cdot 2^{k+1}   \xrightarrow {~ g ~ }  47 \cdot 2^{k+1}   \xrightarrow {~ g ~ }  35 \cdot 2^{k+2}  \xrightarrow {~ g ~ }  47\cdot 2^{k+2}   \xrightarrow {~ g ~ }  \dots$$
More generally, any such prime would give rise to the orbit
$$ (8m + 7) \cdot 2^k  \xrightarrow {~ g ~ }  (6m+5) \cdot 2^{k+1}   \xrightarrow {~ g ~ }  (8m+7) \cdot 2^{k+1}   \xrightarrow {~ g ~ }  (6m+5) \cdot 2^{k+2}    \dots$$
A computer search shows that no such prime $p = 8m+7$ exists for $p \leq 10^{10}$.

  \end{document}